\documentclass[12pt]{article}
\usepackage[a4paper, margin=2.3cm]{geometry}
\usepackage{amsmath,amssymb,amsthm}
\usepackage{color}
\usepackage{parskip}
\newtheorem{theorem}{Theorem}[section]
\newtheorem{remark}{Remark}[section]

\newtheorem{corollary}[theorem]{Corollary}

\newtheorem{lemma}[theorem]{Lemma}
\newtheorem*{definition*}{Definition}

\newcommand\be{\begin{eqnarray*}}
\newcommand\ee{\end{eqnarray*}}
\newcommand\beq{\begin{equation}}
\newcommand\eeq{\end{equation}}

\newcommand\ben{\begin{eqnarray}}
\newcommand\een{\end{eqnarray}}

\def\E{\mathcal{E}}

\def\RR{\mathcal{R}}
\def\S{\mathcal{S}}
\newcommand{\I}{\mathcal{I}}

\begin{document}
\title{A new perspective on the distance problem over prime fields}
\author{Alex Iosevich\thanks{Department of Mathematics, University of Rochester. Email: {\tt iosevich@math.rochester.edu}}\and
Doowon Koh \thanks{Department of Mathematics, Chungbuk National University. Email: {\tt koh1313@gmail.com}}
\and 
    Thang Pham\thanks{Department of Mathematics,  University of Rochester. Email: {\tt vpham2@math.rochester.edu}}}
    \maketitle
\begin{abstract}
Let $\mathbb{F}_p$ be a prime field, and $\E$ a set in $\mathbb{F}_p^2$. Let $\Delta(\E)=\{||x-y||: x,y \in \E\}$, the distance set of $\E$. In this paper, we provide a quantitative connection between the distance set $\Delta(\E)$ and the set of rectangles determined by points in $\E$. As a consequence, we obtain a new lower bound on the size of $\Delta(\E)$ when $\E$ is sufficiently small, improving a previous estimate due to Lund and Petridis and establishing an approach that should lead to significant further improvements. 
\end{abstract}
\section{Introduction}

Let $\mathbb{F}_p$ be the prime field of order $p$. Given two points $x=(x_1, x_2)$ and $y=(y_1, y_2)$ in $\mathbb{F}_p^2$, the distance between $x$ and $y$ is defined by 
\[||x-y||:=(x_1-y_1)^2+(x_2-y_2)^2.\]
For $\E\subset \mathbb{F}_p^2$, define
$$ \Delta(\E):=\{||x-y||: x,y \in \E \}.$$ 
Bourgain, Katz, and Tao \cite{bkt} proved that if $|\E|=p^\alpha$, $0<\alpha<2$, and $p\equiv 3\mod 4,$ then 
\begin{equation}\label{khoidau} |\Delta(\E)| \gg|\E|^{\frac{1}{2}+\epsilon}\end{equation}
for some $\epsilon=\epsilon(\alpha)>0$, where here and throughout, $X \gg Y$ means that there exists a uniform constant $C$ such that $X \ge CY$. 

The exponent $\frac{1}{2}+\epsilon$ has been quantified and improved over the years. Stevens and De Zeeuw \cite{ffrank} used a new point-line incidence bound and the framework in \cite{bkt} to show that 
\begin{equation}\label{S-Z} |\Delta(\E)| \gg |\E|^{\frac{1}{2}+\frac{1}{30}}=|\E|^{\frac{8}{15}} \end{equation}
under the condition that $|\E| \ll p^{\frac{15}{11}}$ with $p\equiv 3\mod 4.$

Iosevich, Koh, Pham, Shen, and Vinh \cite{I-P} improved this result by using a lower bound on the number of distinct distances between a line and a set in $\mathbb{F}_p^2$, and the additive energy of a set on the paraboloid in $\mathbb F_p^3.$ They proved that for $\E\subset \mathbb{F}_p^2$ with $p\equiv 3\mod 4$, if $|\E|\ll p^{\frac{7}{6}}$, then 
$$ |\Delta(\E)| \gtrsim |\E|^{\frac{1128}{2107}}=|\E|^{\frac{1}{2} + \frac{149}{4214}}.$$
This result has been recently improved by Lund and Petridis in \cite{Ben}, namely, they indicated that 
$$ |\Delta(\E)| \gtrsim |\E|^{\frac{1}{2} + \frac{3}{74}},$$
under the condition that $|\E|\ll p^{8/5}$ with $p\equiv 3\mod 4.$
%and $\E$ is not contained fully in any line. 

%Here and throughout, $X \ll Y$ means that there exists $c_1>0$, independent of $p$, such that $X \leq c_1Y$, $X \gtrsim Y$ means $X\gg (\log Y)^{-c_2} Y$ for some positive constant $c_2$, and $X\sim Y$ means that $c_3X\le Y\le c_4X$ for some positive constants $c_3$ and $c_4$.
Here and throughout, $X \gtrsim Y$ means that for every $\epsilon >0$, there exists $C_\epsilon$ such that $ X \le C_\epsilon 
q^\epsilon Y$, and $X\sim Y$ means that $c_3 X\le Y\le c_4X$ for some positive constants $c_3$ and $c_4$, independent of $p.$

\vskip.125in 

\begin{remark} In the setting of arbitrary finite fields $\mathbb{F}_q$, Iosevich and Rudnev \cite{ir} showed that the conclusion (\ref{khoidau}) does not hold. For instance, assume that $q=p^2$, then one can take $\E=\mathbb{F}_p^2$, then it is not hard to see that $|\Delta(\E)|=|\mathbb{F}_p|=|\E|^{1/2}$. Thus, they reformulated the problem in the spirit of the Falconer distance conjecture in the Euclidean space. More precisely, for $\E\subset \mathbb{F}_q^d$, how large does $\E$ need to be to guarantee that $\Delta(\E)$ covers either the whole field or a positive proportion of all elements in $\mathbb{F}_q$? 

Using Fourier analytic methods, Iosevich and Rudnev \cite{ir} proved that for $\E\subset \mathbb{F}_q^d$, if $|\E|\ge 4q^{\frac{d+1}{2}}$, then $\Delta(\E)={\Bbb F}_q$.  It has been shown in \cite{hart} that the exponent $(d+1)/2$ cannot in general be improved when $d$ is odd, even if we only want to recover a positive proportion of all the distances. However, in even dimensional spaces, it has been conjectured that the exponent $(d+1)/2$ can be decreasing to $d/2$, which is in line with the Falconer distance conjecture in the Euclidean space. Chapman, Erdogan, Koh, Hart and Iosevich \cite{CEHIK10} proved that if $\E\subset \mathbb{F}_q^2$, $q$ is prime, $q \equiv 1 \mod 4$ and $|\E|\ge q^{4/3}$, then $|\Delta(\E)|\gg q.$ In the process, they showed that 
if $Cq\le |\E| \le q^{\frac{4}{3}}$ for a sufficiently large constant $C$, then $|\Delta(\E)| \gg \frac{|\E|^{3/2}}{q}$. This result was generalized to arbitrary finite fields in \cite{BHIRP}. See also \cite{KPV} for some recent progress on related problems. 
\end{remark} 

\vskip.125in 

The main purpose of this paper is to provide a connection between the distance set $\Delta(\E)$ and the set of rectangles determined by points in $\E$ in the plane over prime fields, and use this paradigm to improve the known exponents. In particular, we obtain a new lower bound on the size of $\Delta(\E)$ when $\E$ is not too large. 

For $a, b, c\in \mathbb{F}_p^2$, we say that the triple $(a, b, c)$ is a \textit{corner} if 
\[(b-a)\cdot (c-a)=0.\]
For $a, b, c, d\in \mathbb{F}_p^2$, we say that the quadruple $(a, b, c, d)$ forms a \textit{rectangle} if the triples $(a, b, d), (b, a, c), (c, b, d)$, and $(d, a, c)$ are corners. The rectangle $(a, b, c, d)$ is called \textit{non-degenerate} if not all its vertices  are on a line.
%if its vertices do not lie on the same line. 
For $\E\subset \mathbb{F}_p^2$, let $\square(\E)$ be the number of non-degenerate rectangles determined by points in $\E$.

Our main result is the following.
\bigskip
\begin{theorem}\label{thm1}
Let $\E$ be a set in $\mathbb{F}_p^2$ with $p\equiv 3\mod 4$. Suppose that $|\E|\ll p^{8/5}$, then 
$$ T(\E)\ll |\E|^2\log{|\E|}+|\Delta(\E)|^{\frac{4}{15}}|\E|^{\frac{33}{15}}+|\E|^{\frac{5}{3}}|\Delta(\E)|^{\frac{4}{15}}\square(\E)^{\frac{4}{15}},$$ where $T(\E)$ denotes the number of isosceles triangles determined by $\E$ and $\square(\E)$ is the number of non-degenerate rectangles determined by $\E$, as above. 

Consequently, 
\[|\Delta(\E)|\gg \min \left\lbrace |\E|^{\frac{12}{19}}, ~\frac{|\E|^{\frac{20}{19}}}{\square(\E)^{\frac{4}{19}}}\right\rbrace.\]
\end{theorem}

\bigskip

In order to use Theorem \ref{thm1} to full effect, we need the following upper bound on $\square(\E)$ due to Lewko. 

\begin{lemma}[\cite{Lew}, Theorem 4]\label{x-1755}
Let $\E$ be a set in $\mathbb{F}_p^2$ with $p\equiv 3\mod 4$. Suppose that $|\E|\ll p^{26/21}$, then we have 
\[\square(\E)\lesssim |\E|^{\frac{99}{41}}.\]
\end{lemma}

As a consequence of Theorem \ref{thm1}, we obtain a new lower bound on the size of $\Delta(\E)$ when $\E$ is sufficiently small.
\bigskip 
\begin{corollary}\label{thmz}
Let $\E$ be a set in $\mathbb{F}_p^2$ with $p\equiv 3\mod 4$. Suppose that $|\E|\ll p^{\frac{1558}{1489}}$, then we have 
\[|\Delta(\E)|\gtrsim |\E|^{\frac{424}{779}}= |\E|^{\frac{1}{2}+\frac{69}{1558}}.\]
\end{corollary}

We note that the condition $|\E|\ll p^{26/21}$ is replaced by $|\E|\ll p^{\frac{1558}{1489}}$, since in the range $p^{\frac{1558}{1489}}\le |\E|\le p^{\frac{26}{21}}$, the bound $p^{-1}|\E|^{3/2}$ due to Chapman et al. \cite{CEHIK10} is better. 

\medskip

\begin{remark} We observe that our exponent in Corollary \ref{thmz} is ${1}/{2}+{69}/{1558} \approx 0.54428$, and the exponent due to 
Lund and Petridis in \cite{Ben} is $1/2+3/74\approx 0.54054.$  It is important to note that our method is completely different. Moreover, if the Lewko bound $\square(\E)\lesssim |\E|^{\frac{99}{41}}$ is improved to the conjectured bound $\square(\E)\lesssim |\E|^{2}$, then the exponent in Corollary \ref{thmz} would improve to $\approx .6315$. Further progress would result from improving the point line incidence bound used in the proof of Theorem \ref{thm1}. The conjectured point line incidence bound combined with the conjectured bound on $\square(\E)$ would improve the exponent in Corollary \ref{thmz} to $\frac{3}{4}$. This is the limitation of our method. 
\end{remark} 

\vskip.125in 

\begin{remark} In the case of finite subsets of ${\Bbb R}^2$, the sharp bound on the number of rectangles was established by Pach and Sharir (\cite{pach}). This raises the possibility of using this approach in the continuous Euclidean setting. We shall address this issue in the sequel. 
\end{remark} 
\bigskip

%\begin{remark}
%We note that in the setting of $\mathbb{R}$, if we use the Szemer\'{e}di-Trotter theorem in the proof of Lemma \ref{ben} and the optimal upper bound on the number of rectangles due to Pach and Sharir \cite{pach}, then our method says that any $N$ points in $\mathbb{R}^2$ determines at least $N^{3/4}$ distinct distances.
%\end{remark}
\bigskip
Assuming that $\E=A\times A \subset \mathbb{F}_p^2$ has Cartesian product structures, Petridis \cite{Pe} used the point-plane incidence bound due to Rudnev \cite{R} to prove that 
\[|\Delta(A\times A)|=|(A-A)^2+(A-A)^2|\gg |A|^{3/2},\]
under the assumption $|A|\le p^{2/3}$. This result has been extended to all dimensions by Pham, Vinh and De Zeeuw \cite{pham}. 

In the following theorem, we will break the exponent $3/2$ for a variant of the distance function, namely, $(A-A)^2-(A-A)^2$ instead of $(A-A)^2+(A-A)^2$. 
\bigskip
\begin{theorem}\label{thm2}
For $A\subset \mathbb{F}_p$ with $|A|\ll p^{\frac{71}{125}}$, we have 
\[|(A-A)^2-(A-A)^2|\gtrsim |A|^{\frac{3}{2}+\frac{1}{142}}.\]
\end{theorem}
In the proof of  Theorem \ref{thm2}, the following theorem which is interesting on its own plays the key role. 
\bigskip
\begin{theorem}\label{pham2}
Let $A$ be a set in $\mathbb{F}_p$. Suppose that $|A||A-A||A^2-A^2|\le p^2$ and $|A-A|=|A|^{1+\epsilon}$ with $0<\epsilon<1/54.$ Then we have
\[|A^2-A^2|\gtrsim |A|^{1+\frac{9-27\epsilon}{17}}.\]
\end{theorem}
To prove Theorem \ref{pham2}, we use a new sum-product idea which has been introduced recently by Rudnev, Shakan, and Shkredov \cite[Theorem $3$]{R1}. Note that in \cite{R1} Rudnev et al. proved that for sufficiently small $A$ in $\mathbb{F}_p$, if $|AA|\le |A|^{1+\epsilon}$ for some positive small $\epsilon$, then $|A-A|\ge |A|^{\frac{3}{2}+ c(\epsilon)}$. As a consequence of this result, they derived that $|AA-AA|\gg |A|^{\frac{3}{2}+\frac{1}{56}}$. However, their method does not imply the same result when we replace $A-A$ by $A+A$. Therefore, breaking the exponent $3/2$ on the size of the set $AA+AA$ or the set $(A-A)^2+(A-A)^2$ is still an open question. We refer the interested reader to \cite{R1} for more discussions.
\paragraph{Acknowledgments:} A. Iosevich was partially supported by the NSA Grant H98230-15-1-0319. D. Koh was supported by Basic Science Research Programs through National Research Foundation of Korea funded by the Ministry of Education (NRF-2018R1D1A1B07044469). T. Pham was supported by Swiss National Science Foundation
grant P400P2-183916. 
\section{Proof of Theorem \ref{thm1}}
To prove Theorem \ref{thm1}, we will make use of the following lemmas.

Let $\E$ be a set in $\mathbb{F}_p^2$. For any two points $a, b\in \E$ with $\|a-b\|\ne 0$, we define $l_{ab}$ as the line defined by the equation $\|x-a\|=\|x-b\|;$ namely,
\begin{equation}\label{eq08}x\cdot 2(b-a)=||b||-||a||.\end{equation}
%where $x=(x_1, x_2), b=(b_1, b_2), a=(a_1, a_2)$. 
This line is the bisector of the line segment joining  points $a$ and $b$ in $\mathbb{F}_p^2.$

We observe that if there is a pair $(c, d)\in \E\times \E$ such that 
\[(d-c, ||d||-||c||)=\lambda\cdot (b-a, ||b||-||a||)\]
for some $\lambda\ne 0$, then the lines $l_{ab}$ and $l_{cd}$ are the same. 

For $\E\subset\mathbb{F}_p^2$, we define
\[Q(\E):=\{(a, b, c, d)\in \E^4\colon l_{ab}=l_{cd}\}.\] 
For $\E\subset \mathbb{F}_p^2$, let  $L$ be the multi-set of lines $l_{ab}$ with $a, b\in \E$ and $\|a-b\|\ne 0.$ 
Let $T(\E)$ be the number of isosceles triangles determined by points in $\E$. In other words, we have
\begin{equation}\label{defTE} T(\E)=|\{(a,b,c)\in \E^3: \|a-c\|=\|b-c\|\ne 0\}|.\end{equation}  
It is clear that $T(\E)$ is bounded by the number of incidences between $\E$ and $L$. 

Using the point-line incidence bound due to Stevens and De Zeeuw \cite{ffrank}, Lund and Petridis proved the following lemma. 
\bigskip
\begin{lemma}[\cite{Ben}, Proof of Lemma 12]\label{ben}
Let $\E$ be a set in $\mathbb{F}_p^2$ with $|\E|\ll p^{8/5}.$ Then we have
\[T(\E)\ll |\E|^2\log |\E|+|\E|^{5/3}|Q(\E)|^{4/15}.\]
\end{lemma}
In our next lemma, we give an upper bound of $|Q(\E)|$ in terms of $|\Delta(\E)|$ and $\square(\E)$. 
\bigskip
\begin{lemma}\label{hay}
For $\E\subset \mathbb{F}_p^2$, we have 
\[|Q(\E)|\ll |\Delta(\E)|\left(\square(\E)+|\E|^2\right).\]
\end{lemma}
\begin{proof}
We partition the set $\{(a, b)\in \E\times \E\colon \|a-b\|\ne 0\}$ into subsets $\{S_i\}_i$ in a way that in each set $S_i$, if  $(a, b), (c, d)\in S_i,$ then $l_{ab}=l_{cd}$. Suppose that there are $n$ such sets $S_i$. 

It is not hard to see that 
$$|Q(\E)|= \sum_{i=1}^n|S_i|^2.$$ We now bound the size of $S_i$ as follows. 

Fix an element $(a, b) \in S_i$. We now partition the set $S_i$ into subsets $\{S_{i\lambda}\}_{\lambda\in \Lambda_i \subset \mathbb F_p^*}$ in a way that if $\lambda\in \Lambda_i$ and $(c, d)\in S_{i\lambda}$, then 
$$\lambda\cdot (2(b-a), ||b||-||a||)=(2(d-c), ||d||-||c||).$$ Note that $\lambda\cdot (2(b-a), ||b||-||a||)=(2(d-c), ||d||-||c||)$ is equivalent to $\lambda\cdot (b-a, ||b||-||a||)=(d-c, ||d||-||c||)$

Suppose that for each $i$ there are $m_i$ such subsets. Namely, $|\Lambda_i|=m_i$  for $i=1,2,\ldots, n.$

We observe that if $(u, v)\in S_{i\lambda}$ and $(c, d)\in S_{i\lambda'}$, then 
$$(v-u, ||v||-||u||)=(v, ||v||)-(u, ||u||)=(\lambda/\lambda')\cdot (d-c, ||d||-||c||)=(\lambda/\lambda')\cdot \left((d, ||d||)-(c, ||c||)\right).$$

We now show that $m_i\ll |\Delta(\E)|$. Indeed, assume that $||a-b||=1$, then for any pair $(c, d)\in S_{i\lambda}$ with $\lambda\in \Lambda_i$,  we have $||c-d||=\lambda^2\in \Delta(\E)$. Therefore, $m_i\ll |\Delta(\E)|$. 

%In other words, 
%\[|S_i|^2=\sum_{\lambda, \lambda'}|S_{i\lambda}||S_{i\lambda'}|\le \left(\sum_{\lambda}|S_{i\lambda}|\right)^2\ll |\Delta(\E)|\sum_{\lambda}|S_{i\lambda}|^2.\]
Hence,  it follows by the Cauchy-Schwarz inequality that 
\[|S_i|^2=\left(\sum_{\lambda\in \Lambda_i}|S_{i\lambda}|\right)^2\ll |\Delta(\E)|\sum_{\lambda\in \Lambda_i}|S_{i\lambda}|^2.\]
On the other hand, we observe that $\sum_{1\le i\le n}\sum_{\lambda\in \Lambda_i}|S_{i\lambda}|^2$ is equal to  the number of quadruples $(a, b, c, d)\in \E^4$ such that $a\ne b, c\ne d$ and $(b-a, ||b||-||a||)=(d-c, ||d||-||c||)$. We note that the relation 
\[(b-a, ||b||-||a||)=(d-c, ||d||-||c||)\]
is equivalent to 
\begin{equation}\label{role}(b, ||b||)-(a, ||a||)+(c, ||c||)=(d, ||d||).\end{equation}
From this equation, we have 
\[||b||+||c||-||a||=||b+c-a||.\]
This can be rewritten as 
\[(b-a)\cdot (c-a)=0.\]
Thus, $(a, b, c)$ is a corner. If we switch the roles of $(a, ||a||), (b, ||b||), (c, ||c||),$ and $(d, ||d||)$ in (\ref{role}), then we will be able to show that $(c, a, d)$, $(d, c, b)$, and $(b, d, a)$ are also corners. This means that $(b, a, c, d)$ is a rectangle. 

We note that if $(a, b)=(c, d)$ then the rectangle $(b, a, c, d)$ is degenerate. However, the number of such degenerate rectangles is at most $|\E|^2$.  In other words, we have proved that 
\[\sum_{1\le i\le n}\sum_{\lambda\in \Lambda_i}|S_{i\lambda}|^2\le\square(\E)+|\E|^2.\]
This completes the proof of the lemma.
\end{proof}
\paragraph{Proof of Theorem \ref{thm1}:}
For $t\in \mathbb{F}_p$, let $\nu(t)$ be the number of pairs $(x, y)\in \E\times \E$ such that $||x-y||=t$. By the Cauchy-Schwarz inequality (or see the proof of Theorem 1.1 in \cite{I-P}), we have 
\begin{equation}\label{Doowon1}\sum_{t\in \mathbb{F}_p}\nu(t)^2\le |\E| |\{(a,b,c)\in \E^3: \|a-b\|=\|a-c\|\}| =|\E| \cdot T(\E),\end{equation}
where the last equality follows since $p\equiv 3\mod 4,$ by assumption, and the definition of $T(\E)$ given in \eqref{defTE}.

It follows from Lemmas \ref{ben} and \ref{hay} that 
\begin{equation}\label{Doowon2}T(\E)\ll |\E|^2\log{|\E|}+|\Delta(\E)|^{\frac{4}{15}}|\E|^{\frac{33}{15}}+|\E|^{\frac{5}{3}}|\Delta(\E)|^{\frac{4}{15}}\square(\E)^{\frac{4}{15}}.\end{equation}
Combining (\ref{Doowon1}) and (\ref{Doowon2}), we have 
\[\sum_{t\in \mathbb{F}_p}\nu(t)^2\ll |\E|^3\log{|\E|}+|\Delta(\E)|^{\frac{4}{15}}|\E|^{\frac{48}{15}}+|\E|^{\frac{8}{3}}|\Delta(\E)|^{\frac{4}{15}}\square(\E)^{\frac{4}{15}}.\]
Applying the Cauchy-Schwarz inequality and using the above inequality, we have 
\[|\Delta(\E)|\gg \frac{|\E|^4}{\sum_{t\in \mathbb{F}_p}\nu(t)^2} \gg \frac{|\E|^4}{|\E|^3\log{|\E|}+|\Delta(\E)|^{\frac{4}{15}}|\E|^{\frac{48}{15}}+|\E|^{\frac{8}{3}}|\Delta(\E)|^{\frac{4}{15}}\square(\E)^{\frac{4}{15}} } .\]
Solving this inequality, we get the desired result. $\hfill\square$
\paragraph{Proof of Corollary \ref{thmz}:} The proof follows directly from Theorem \ref{thm1} and Lemma \ref{x-1755}. 
\bigskip

\section{Proof of Theorem \ref{thm2}}
To prove Theorem \ref{thm2}, we will make use of the following results. The first lemma was given by Pham, Vinh and De Zeeuw in \cite{pham}. 
\bigskip
\begin{lemma}[\cite{pham}, Corollary 3.1]\label{pham-1}
Let $A, X$ be sets in $\mathbb{F}_p$ with $|X|\gg |A|.$ Then we have
\[|X-(A-A)^2|\gg \min\{p, |X|^{1/2}|A|\}.\]
\end{lemma}
\bigskip
Notice that our next theorem is a more general form of Theorem \ref{pham2}.
A proof of the following theorem will be given  after proving Theorem \ref{thm2}.
\bigskip
\begin{theorem}\label{lmchinh}
Let $A$ be a set in $\mathbb{F}_p$. Suppose that $|A||A- A||A^2-A^2|\le p^2$, $|A- A|=M|A|$, and $|A^2-A^2|=K|A|.$ Then we have that $M^{27}K^{17}\gtrsim |A|^9$ or $M^{13}K^7\gtrsim |A|^5$.
\end{theorem}
\paragraph{Proof of Theorem \ref{thm2}:}
By a translation if necessary, we assume that $0\in A$. Let $\epsilon>0$ be a parameter chosen at the end of the proof. 

 If $|A-A|\ge |A|^{1+\epsilon}$, then it follows from Lemma \ref{pham-1} that 
\begin{equation}\label{PhamC1}|(A-A)^2-(A-A)^2|\gg |A|^{\frac{3}{2}+\frac{\epsilon}{2}},\end{equation}
provided that $|A|\le p^{2/(3+\epsilon)}.$

On the other hand, if $|A-A|\le |A|^{1+\epsilon},$ then we now fall into two cases:

{\bf Case $1$:} If $|A-A||A||A^2-A^2|\ge  p^2$, then we have 
\begin{equation}\label{PhamC2}|(A-A)^2-(A-A)^2|\ge |A^2-A^2| \ge \frac{p^2}{|A-A||A|} \ge \frac{p^2}{|A|^{2+\epsilon}},\end{equation}
where the first inequality above holds since $0\in A.$

%\[K\ge \frac{p^{2}}{|A|^{3+\epsilon}}.\]
%Since $0\in A$, we have
%\[|(A-A)^2-(A-A)^2|\gg |A^2-A^2|\gg |A|K\ge \frac{p^2}{|A|^{2+\epsilon}}.\]

{\bf Case $2$:} If $|A-A||A||A^2-A^2|\le p^2$, then it follows from Theorem \ref{pham2} that $|A^2-A^2|\gtrsim |A|^{1+\frac{9-27\epsilon}{17}}$. Since $0\in A$, we have
\begin{equation}\label{PhamC3}|(A-A)^2-(A-A)^2|\ge |A^2-A^2|\gtrsim  |A|^{1+\frac{9-27\epsilon}{17}}.\end{equation}

We choose $\epsilon=1/71.$ Then we obtain the required conclusion in the cases of \eqref{PhamC1} and \eqref{PhamC3}. Choosing $\epsilon=1/71,$ the case of \eqref{PhamC2} also gives the desirable consequence since we have
 \[|(A-A)^2-(A-A)^2|\ge \frac{p^2}{|A|^{2+\epsilon}}\ge |A|^{\frac{3}{2}+\frac{1}{142}},\]
under the condition $|A|\le p^{71/125}.$ This completes the proof of the theorem. $\hfill\square$
%\[|(A-A)^2-(A-A)^2|\gg |A^2-A^2|\gg |A|^{\frac{3}{2}+\frac{1}{142}}.\]
%We also have 

\subsection{Proof of Theorem \ref{lmchinh}}
In the proof of Theorem \ref{lmchinh}, we will use the point-plane incidence bound due to Rudnev \cite{R}, but
we use a strengthened version of this theorem, proved by de Zeeuw in \cite{Z}. Let us first recall that if $\RR$ is a set of points in $\mathbb{F}_p^3$ and $\S$ is a set of planes in $\mathbb{F}_p^3$, then the number of incidences between $\RR$ and $\S$, denoted by $I(\RR, \S)$, is the cardinality of the set $\{(r,s)\in \RR\times \S : r\in s\}$.
\bigskip
\begin{theorem}[{\bf Rudnev}, \cite{R}]\label{thm:rudnev}
Let $\RR$ be a set of points in $\mathbb{F}_p^3$ and $\S$ be a set of planes in $\mathbb{F}_p^3$, with $|\RR|\leq |\S|$.
Suppose that there is no line that contains $k$ points of $\RR$ and is contained in $k$ planes of $\S$.
Then
\[ \I(\RR,\S)\ll\frac{|\RR||\S|}{p}+ |\RR|^{1/2}|\S| +k|\S|.\]
\end{theorem}
In this section, we assume that $|A-A|=M|A|$ and $|A^2-A^2|=K|A|$. For $A\subset \mathbb{F}_p$, we define $E_4(A^2)$ as the number of $8$-tuples $(x_1, x_2, x_3, x_4, y_1, y_2, y_3, y_4)\in (A^2)^8$ such that
\[x_1-y_1=x_2-y_2=x_3-y_3=x_4-y_4.\]
In the following lemma, we give an upper bound of $E_4(A^2)$ in terms of $|A|$ and $|A-A|$.
\bigskip 
\begin{lemma}\label{e4}
Let $A$ be a set in $\mathbb{F}_p$. Suppose that $|A-A||A||A^2-A^2|\ll p^2$, then we have
\[E_4(A^2)\lesssim |A|^4M^3.\]
\end{lemma}
\begin{proof} For each $x\in \mathbb F_p,$  we define $r_{A^2-A^2}(x)$ as the number of pairs $(y,z)\in A^2\times A^2$ such that $y-z=x.$
For $1\le k\ll |A|$, let 
\[n_k:=|X_k:=\{x\in A^2-A^2\colon r_{A^2-A^2}(x)\ge k\}|.\]
By a dyadic pigeon-hole argument, we have
\[E_4(A^2)\ll \sum_{k}n_k\cdot k^4.\]
Thus,  it is enough to show that 
\[n_k\lesssim \frac{M^3|A|^4}{k^4}.\]
To this end, we consider the following equation 
\begin{equation}\label{eq1115}(x+u)^2-y^2=t\end{equation}
with $x\in A-A, u\in A, y\in A, t\in X_k$.  Let $N$ be the number of solutions of this equation.  It is clear that $N\ge k|X_k||A|$. By the Cauchy-Schwarz inequality, we have 
\begin{align*}
N&\le |A|^{1/2}\sqrt{|\{(x, u, t, x', u', t')\in ((A-A)\times A\times X_k)^2\colon (x+u)^2-t=(x'+u')^2-t'\}|}\\&=:|A|^{1/2}\sqrt{N'}.
\end{align*}
To bound $N'$, we let $\RR$ be the set of points of the form $(2x, u', -t+x^2-u'^2)$ with $x\in A-A, u'\in A, t\in X_k$. Let $\S$ be the set of planes of the form $uX-(2x')Y+Z=x'^2-u^2-t'$ with $u\in A, x'\in A-A, t'\in X_k$. We have $|\RR|=|\S|\sim |A-A||A||X_k|$. It is not hard to see that there are at most $|A-A|$ collinear points in $\RR$ except that there are vertical lines supporting $|X_k|$ points, but planes in $\S$ contain no vertical lines. Thus, we can apply Theorem \ref{thm:rudnev} with $k=|A-A|$ to get 
\begin{equation}\label{RHS} N'=I(\RR, \S)\ll |A-A|^{3/2}|A|^{3/2}|X_k|^{3/2}+|A-A|^2|A||X_k|,\end{equation}
where we also use the assumption  $|A-A||A||A^2-A^2|\ll p^2$ and the fact that $|X_k|\le |A^2-A^2|.$
If the second term of the RHS of the inequality \eqref{RHS}  dominates, then we get 
\[|X_k|\le \frac{|A-A|}{|A|}=M.\]
So, $E_4(A^2)\lesssim M|A|^4$. 

If the first term dominates, then from the inequalities $k|X_k||A|\le N\le |A|^{1/2}(N')^{1/2}$ we have 
\[|X_k|\le \frac{|A|^4M^3}{k^4}.\]
With this upper bound of $|X_k|$, we have 
\[E_4(A^2)\lesssim |A|^4M^3.\]
This completes the proof of the lemma. 
\end{proof}

Let $P$ be the subset of $A^2-A^2$ such that for any $x\in P$, we have $r_{A^2-A^2}(x)\ge \frac{|A|}{2K}$. For any $w\in A^2-A^2$, let $P_w:=(A^2-A^2)\cap (P-w)$. One can follow the first paragraph of the proof of \cite[Theorem $3$]{R1} to prove the following lemma. 
\bigskip
\begin{lemma}\label{eq999}
For $A\subset \mathbb{F}_p$, we have
\begin{equation*}|A|^4\ll \sqrt{E_4(A^2)}\sqrt{\mathcal{X}},\end{equation*}
where 
\[\mathcal{X}=\sum_{w\in A^2-A^2}|\{(u, v)\in P_w\times P_w\colon u-v\in P\}|.\]
\end{lemma}
\begin{proof}
We consider the following equation:
\begin{equation}\label{eq-Thang}x-u=y-v=w,\end{equation}
with $x, y\in P$, $u, v, w\in A^2-A^2$. It follows from the definition of $P$ that 
\[|\{(a_1, a_2)\in A\times A\colon a_1^2-a_2^2\in P\}|\gg |A|^2.\] 
One can use the dyadic pigeon-hole to show that there exits a subset $A'\subset A$ with $|A'|\gtrsim |A|$ such that for any $y\in A'$ the number of $x\in A$ such that $x^2-y^2\in P$ is at least $\gg |A|$. For each $y\in A'$, we denote the set of such $x$ by $N_y$. 

For any $c\in A$ and $b\in A'$, we have 
\begin{equation}\label{eq-Thang2}c^2-b^2=(a^2-b^2)-(a^2-c^2)=(d^2-b^2)-(d^2-c^2).\end{equation}
Thus, $(x, y, u, v)=(a^2-b^2,  d^2-b^2, a^2-c^2, d^2-c^2)$, with $b\in A', a, d\in N_b, c\in A$, is a solution of (\ref{eq-Thang}). In other words, there are at least $\gtrsim |A|^4$ tuples $(a, b, c, d)\in A^4$ which gives us solutions of (\ref{eq-Thang}) in the form of (\ref{eq-Thang2}).

For each tuple $(a, b, c, d)\in A^4$, we define $[a^2, b^2, c^2, d^2]$ as the set of tuples $(a', b', c', d')\in A^4$ such that $(a^2, b^2, c^2, d^2)=(a'^2, b'^2, c'^2, d'^2)+(t, t, t, t)$ for some $t\in \mathbb{F}_p$. It is not hard to check that this defines an equivalence class by translation. On the other hand, each equivalence class gives us an unique solution $(x, y, u, v)$ of (\ref{eq-Thang}). Therefore, by the Cauchy-Schwarz inequality, we have 
\begin{align*}
|A|^4&\lesssim \sqrt{\sum_{[a^2, b^2, c^2, d^2]}|[a^2, b^2, c^2, d^2]|^2}\sqrt{|\{(x, y, u, v)\in P^2\times (A^2-A^2)^2\colon x-u=y-v=w\}|}\\
&\lesssim \sqrt{E_4(A^2)}\sqrt{|\{(x, y, u, v)\in P^2\times (A^2-A^2)^2\colon x-u=y-v=w\}|}\\
&=\sqrt{E_4(A^2)} \sqrt{\mathcal{X}}.
\end{align*}
This completes the proof of the lemma.
\end{proof}
In the next step, we need to bound $\mathcal{X}$. To this end, we need the following lemmas.
\bigskip
\begin{lemma}\label{lm1}
For any $w\in A^2-A^2$, suppose that $|A-A||A||A^2-A^2|\ll p^2.$ Then we have
\[T_w:=|\{(u, v)\in P_w\times P_w\colon u-v\in P\}|\ll M^{3/2}K|P_w|^{3/2}+M^2K|P_w|.\]
\end{lemma}
\begin{proof}
For any $p\in P$, we have $r_{A^2-A^2}(p)\gg \frac{|A|}{K}$, so 
\[T_w\cdot \frac{|A|}{K}\cdot |A|^2\ll |\{(x, y, z, t, u, v)\in (A-A)^2\times A^2\times P_w^2\colon (x+z)^2+(y+t)^2=u-v\}|.\]
By repeating the argument of the proof of Lemma \ref{e4} with $P_w$ in the place of $X_k$, we have the number of such tuples $(x, y, z, t, u, v)$ is bounded by 
\[|A-A|^{3/2}|P_w|^{3/2}|A|^{3/2}+|A||P_w||A-A|^2\ll M^{3/2}|A|^3|P_w|^{3/2}+M^2|A|^3|P_w|.\]
This gives us 
\[T_w\ll M^{3/2}K|P_w|^{3/2}+M^2K|P_w|.\]
\end{proof}
Suppose we sort the set $A^2-A^2$ in the following order $A^2-A^2:=\{w_1, \ldots, w_{|A^2-A^2|}\}$ with $r_{P-(A^2-A^2)}(w_i)\ge r_{P-(A^2-A^2)}(w_j)$ for any $i\ge j$. 
\bigskip
\begin{lemma}\label{lm2}
Suppose that $|A^2-A^2||A-A||A|\le p^2$. For any $1\le n \le |A^2-A^2|$, we have 
\[|P_{w_n}|\le M^{3/2}K^2|A|n^{-1/2}.\]
\end{lemma}
\begin{proof}
Set $W_t:=\{w\in A^2-A^2\colon r_{P-(A^2-A^2)}(w)\ge t\}$, i.e. $W_t$ is the set of elements $w$ such that there are at least $t$ pairs $(x, y)\in P\times (A^2-A^2)$ such that $w=x-y$. It is not hard to check that 
\[t|W_t|\cdot \frac{|A|}{K}\cdot |A|^2\le |\{(w, x, y ,z , v, u)\in W_t\times (A-A)^2\times A^2\times (A^2-A^2)\colon w=(x+z)^2+(y+v)^2-u\}|.\]

We repeat the argument of the proof of Lemma \ref{e4} with $\RR:=\{(2x, v, x^2+v^2-w)\colon x\in A-A, v\in A, w\in W_t\}$ and $\S:=\{zX+2yY+Z=u\colon z\in A, y\in A-A, u\in A^2-A^2\}$ to bound the number of such tuples $(w, x, y, z, v, u)$ by 
\[|A-A|^{3/2}|A|^{3/2}|W_t|^{1/2}|A^2-A^2|+|A-A|^2|A^2-A^2||A|.\]
If the second term dominates, we have 
\[|W_t|\le \frac{|A-A|}{|A|}=M\le \frac{K^4M^3|A|^2}{t^2}.\]
Otherwise, we get 
\[|W_t|^{1/2}\le \frac{K^2M^{3/2}|A|}{t}.\]
We observe that $r_{P-(A^2-A^2)}(w_n)= |P_{w_n}|$. So, 
\[n=|\{w\in A^2-A^2\colon r_{P-(A^2-A^2)}(w)\ge |P_{w_n}| \}|\le \frac{K^4M^3|A|^2}{|P_{w_n}|^2}.\] 
In other words, 
\[|P_{w_n}|\le M^{3/2}|A|K^2n^{-1/2}.\]
This completes the proof of the lemma. 
\end{proof}
\bigskip

\begin{lemma}\label{TTLem}
Let $A$ be a set in $\mathbb{F}_p$. Suppose that $|A^2-A^2||A-A||A|\le p^2.$ Then we have
\[\mathcal{X} \le M^{15/4}K^{17/4}|A|^{7/4}+M^{7/2}K^{7/2}|A|^{3/2}.\]
\end{lemma}
\begin{proof}
Applying Lemmas \ref{lm1} and \ref{lm2}, we have
\begin{align*}
\mathcal{X}&\le \sum_{w_n\in A^2-A^2}|T_{w_n}|\ll M^{3/2}K\sum_{w_n}|P_{w_n}|^{3/2}+M^2K\sum_{w_n}|P_{w_n}|\\
& \le M^{15/4}K^{17/4}|A|^{7/4}+M^{7/2}K^{7/2}|A|^{3/2}.
\end{align*} 
This completes the proof of the lemma.
\end{proof}
\paragraph{Proof of Theorem \ref{lmchinh}:}

It follows from Lemmas \ref{e4} and \ref{eq999} that $|A|^2\lesssim M^{\frac{3}{2}}\mathcal{X}^{1/2}.$  Combining Lemma \ref{TTLem} with the above inequation, we obtain
\[|A|^2\lesssim M^{27/8}K^{17/8}|A|^{7/8}+M^{13/4}K^{7/4}|A|^{3/4}.\]
This implies that $M^{27}K^{17}\gtrsim |A|^9$ or $M^{13}K^7\gtrsim |A|^5,$ 
as desired.$\hfill\square$.

\end{document}